\documentclass[11pt]{nyjm}

\usepackage{amsmath, amssymb, amsthm, amsfonts, tikz, mathtools, enumerate, graphicx, mathrsfs, yhmath, hyperref, enumerate, lipsum, float}
\usetikzlibrary{calc}
\usepackage{cleveref}
\usepackage{doi}
\usepackage{fancybox}
\usepackage{enumitem}
\usepackage{calc}
\usetikzlibrary{cd}

\newtheorem*{theorem*}{Theorem}
\newtheorem{theorem}{Theorem}[section]
\newtheorem{lemma}[theorem]{Lemma}

\theoremstyle{definition}
\newtheorem{defn}[theorem]{Definition}
\newtheorem{prop}[theorem]{Proposition}

\theoremstyle{remark}
\newtheorem{ex}{Example}
\newtheorem{rem}{Remark}

\numberwithin{equation}{section}

\newcommand\blfootnote[1]{%
    \begingroup
    \renewcommand\thefootnote{}\footnote{#1}%
    \addtocounter{footnote}{-1}%
    \endgroup
}
\newcommand{\cst}{\mathrm{C}^{*}} 
\newcommand{\K}{\mathcal{K}}
\newcommand{\B}{\mathcal{B}}
\newcommand{\C}{\mathbb{C}}

\newcommand{\id}{\text{id}}
\newcommand{\la}{\langle}
\newcommand{\ra}{\rangle}
\newcommand{\mf}{\mathfrak}

\newcommand{\ext}{\text{ext}}
\newcommand{\G}{\mathcal{G}}
\newcommand{\Ad}{\text{Ad}}
\newcommand{\h}{\mathfrak{H}}
\newcommand{\nl}{\lVert}

\newcommand{\nr}{\rVert}
\newcommand{\M}{\mathrm{M}}

\newcommand{\spn}{\text{span}}
\newcommand{\Ind}{\text{Ind}}
\newcommand{\what}{\widehat}
\newcommand{\op}{\text{op}}

\renewcommand{\r}{\mathcal{R}}
\renewcommand{\sf}{\mathsf}
\newcommand{\scr}{\mathscr}
\renewcommand{\l}{\left}
\renewcommand{\r}{\right}
\renewcommand{\sf}{\mathsf}
\renewcommand{\H}{\mathcal{H}}
\newcommand{\p}{\prescript}
\renewcommand{\j}{\hat{\jmath}}
\renewcommand{\it}{\textit}
\renewcommand{\L}{\mathcal{L}}
\renewcommand{\o}{\overline}
\makeatletter
\newcommand\thmsname{\protect\theoremname}
\newcommand\nm@thmtype{theorem}
\theoremstyle{plain}

\theoremstyle{plain}

\title{The Ladder Technique - Quantum Groups}

\author[Gillespie]{Matthew Gillespie}
\address{Matthew Gillespie
\\School of Mathematical and Statistical Sciences\\
Arizona State University\\
Tempe, AZ 85287}
\email{mjgille1@asu.edu}
\begin{document}

\begin{abstract}
    Given a regular $\cst$-algebraic locally compact quantum group $(S_r,\Delta)$ with universal quantum group $(S_f,\Delta_f)$, a $\cst$-algebra $A$, and a sufficiently well-behaved full coaction $S_f \overset{\alpha}{\curvearrowright} A$, we construct natural lattice isomorphisms from the strongly coaction invariant ideals of $A$ to the strongly coaction invariant ideals of full and reduced crossed product $\cst$-algebras as an application of the `ladder technique' developed by the author, S. Kaliszewski, John Quigg and Dana P. Williams. In particular, these lattice isomorphisms are determined by either the maximality or normality of the coaction $\alpha$.  This result directly generalizes a recent theorem proven by the aforementioned authors for locally compact groups, which in turn generalized a theorem of Elliot Gootman and Aldo Lazar for amenable groups.  
\end{abstract}

\maketitle


\section{Introduction}
Given a $\cst$-dynamical system $(A,G,\alpha)$ or cosystem $(A,G,\delta)$, one can form the crossed product $\cst$-algebras \begin{align*}
    &A \rtimes_{\alpha} G \: \: \text{and} \: \:
    A \rtimes_{\delta} G
\end{align*} which are the universal objects for covariant representations of the respective systems \cite[Section~A.2,A.5]{enchilada}. For actions specifically, we can also define a reduced crossed product $A \rtimes_{\alpha,r} G$ via quotienting $A \rtimes_{\alpha} G$ by the kernel of the regular representation.  Since these $\cst$-algebras are inherently tied to their original dynamical systems, it is of much interest to understand how the ideal structure of the crossed product $\cst$-algebras 
$A \rtimes_{\alpha} G, A \rtimes_{\alpha,r} G$ and $A \rtimes_{\delta} G$ relate back to the ideals of the original $\cst$-algebra $A$ which are invariant in some sense under the (co)action. Gootman and Lazar resolved this question for amenable groups $G$ in \cite{goot} (with reduced coactions), and later Nilsen in \cite{nil} answered this for arbitrary locally compact $G$ for actions, but under a different form of invariance for (full) coactions.  Recently, this was proven again for arbitrary locally compact $G$ except under strong coaction invariance, see \cite[Theorem~3.2]{ladder}.  This notion of invariance was used in \cite{goot}, and is equivalent to the invariance used in \cite{nil} when $G$ is amenable (the relationship in general has also been worked out in an upcoming paper by the first three authors of \cite{ladder}). The main result of this paper directly extends the result presented in \cite{ladder} for appropriately defined Kac systems coming from a regular $\cst$-algebraic locally compact quantum group $\G$, which we will call a LCQG from now on.  This is accomplished via applying `the ladder technique' developed in \cite{ladder}.  This technique amounted to a basic observation of commuting diagrams involving bijections (see \cite[Lemma~3.1]{ladder}) applied to iterated crossed products of a fixed $\cst$-dynamical (co)system $(A,G,\alpha)$ and $(A,G,\delta)$, where the bijections needed for \cite[Lemma~3.1]{ladder} came from appropriately restricting the lattice isomorphisms gotten via the Rieffel correspondence. Much of the machinery presented here is very similar in nature to that presented in \cite{ladder}, and this is by design.  The main purpose of this paper is to highlight another application of the main techniques in \cite{ladder}.  We begin in section $2$ with some quick notational conventions and identifications.  In section $3$, we give a basic overview of Hopf $\cst$-algebras, weak Kac systems with full duality, multiplicative unitaries, LCQGs, and correspondence coactions in line with \cite[Definition~2.10]{enchilada}.  In section $4$, we discuss crossed product duality, ideal invariance, and the four main propositions needed to prove the main result in section $5$ via applying the `ladder technique'.  I would like to give my sincere thanks to Stefaan Vaes, Alcides Buss and Siegfried Echterhoff for their remarkable patience and many helpful conversations assisting me in learning the basics of this theory.        

\section{Notation \& Conventions}
  We parrot many of the conventions from \cite{fischer}, and insist on only using full coactions.  We do this because we want our coactions to mirror the full coactions one gets from \cite[Definition~A.67]{enchilada}. All $*$-homomorphisms are non-degenerate unless specified otherwise. The symbol $\otimes$ always will mean the minimal tensor product of $\cst$-algebras, and otherwise will be decorated to indicate a $B$-balanced $\otimes_B$ or exterior tensor product $\otimes_{\ext}$ of correspondences.  If $\sf{X}$ is an $(A,B)$-correspondence, we denote the \it{multiplier correspondence} $\M(\sf{X}) \coloneq \mathcal{L}_{B}(B,\sf{X})$.  That is, the $B$-linear adjointable operators $B \longrightarrow \sf{X}$.  Given a Hilbert space $\H$, we will always denote the compact operators on $\H$ by $\K$ when there is no ambiguity of the ambient Hilbert space.  We also will utilize the standard leg-numbering notation where \begin{align*}
      (a \otimes b)_{12} &= (\id \otimes 1)(a \otimes b) = a \otimes b \otimes 1, \\
      (a \otimes b)_{23} &= (1 \otimes \id)(a \otimes b) = 1 \otimes a \otimes b, \\
      (a \otimes b)_{13} &= (\Sigma \otimes \id) \circ (a \otimes b)_{23} = a \otimes 1 \otimes b.
  \end{align*}

\section{Preliminaries}

\subsection{Hilbert Modules \& correspondences}

Hilbert modules and correspondences are an indispensable tool for the study of Morita equivalence and the theory of crossed product $\cst$-algebras.  Several arguments throughout this work will use results found in \cite{enchilada}.  For the sake of self-containment, we opt to include a brief background on these objects here.    

\begin{defn} Let $A,B$ be $\cst$-algebras. \label{def19}
\begin{enumerate}[label=(\alph*)]
    \item A \textit{right-Hilbert $B$-module} is a vector space $\sf{X}_B$ with a positive definite sesquilinear form $\la \cdot,\cdot\ra_B: \sf{X}_B \times \sf{X}_B \longrightarrow B$ such that 
    \begin{enumerate}[label=(\alph*)]
        \item[1.] $\la x,y \cdot b \ra_B = \la x,y \ra_Bb$ for all $x,y \in \sf{X}_B$ and $b \in B$.
        \item[2.] $\la x,y \ra_B^{*} = \la y,x \ra_B$ for all $x,y \in \sf{X}_B$ and $b \in B$.
        \item[3.] $\sf{X}_B$ is a Banach space in the norm $\nl x \nr \coloneq \nl \la x,x \ra_B \nr^{\frac{1}{2}}$.
    \end{enumerate}
    We say that $\sf{X}_B$ is \it{full} if $\o\spn\l\{\la \sf{X},\sf{X}\ra_B\r\}=B$.
    \item An \textit{$(A,B)$-correspondence} is a right-Hilbert $B$-module $_A\sf{X}_B$ which is a non-degenerate left $A$-module ($A \cdot \sf{X} = \sf{X}$) satisfying
    \begin{enumerate}
        \item[1.] $a \cdot (x \cdot b) = (a \cdot x) \cdot b$ for all $x \in \: _{A}\sf{X}_B$, $a \in A$ and $b \in B$.
        \item[2.] $\la a \cdot x,y \ra_B = \la x,a^{*} \cdot y \ra_B$ for all $x,y \in \: _{A}\sf{X}_B$ and $a \in A$.
    \end{enumerate}
\end{enumerate}
\end{defn}

\begin{defn} Let $A,B$ be $\cst$-algebras. \label{def20}
    \begin{enumerate}[label=(\alph*)]
        \item An \it{$(A,B)$-imprimitivity bimodule} is an $(A,B)$-correspondence $\sf{X}$ with left $A$-valued inner product satisfying $_{A}\la x,y \ra \cdot z = x \cdot \la y,z \ra_B$ for all $x,y,z \in \mathsf{X}$ such that $\sf{X}$ is full and $\o\spn\l\{_A\la \sf{X},\sf{X} \ra\r\} = A$.
        \item $A,B$ are called \it{Morita equivalent} if there exists an $(A,B)$-imprimitivity bimodule $_A\sf{X}_B$.
    \end{enumerate}
\end{defn}
\begin{theorem}[{\cite[Proposition~3.24]{mect}}] \label{thmm}  
        If $A,B$ are Morita equivalent with imprimitivity bimodule $_A\mathsf{X}_B$, there is a lattice isomorphism $\mathsf{X}-\text{Ind}^A_B: \mathscr{I}(B) \longrightarrow \mathscr{I}(A)$ by \[\mathsf{X}-\text{Ind}^A_B(J) = \o\spn\l\{\: _{A}\la x \cdot b,y \ra \: | \: x,y \in \mathsf{X}, b \in J\r\}\] where $\scr{I}(A), \scr{I}(B)$ are the lattices of ideals of $A,B$. The lattice isomorphism is called the \it{Rieffel correspondence associated to $\sf{X}$}.
\end{theorem}

Note that Rieffel's theorem actually contains much more detail than the above statement, but this will suffice for us.  We now define the notion of a morphism between correspondences.

\begin{defn} \label{def21}
    Let $_A\sf{X}_B$ and $_C\sf{Y}_D$ be correspondences for $\cst$-algebras $A,B,C,D$.  Moreover, let $\varphi: A \longrightarrow \M(C)$ and $\psi: B \longrightarrow\M(D)$ be $*$-homomorphisms with $\Phi: \sf{X} \longrightarrow \M(\sf{Y})$ a linear map.  We call $\Phi$ a \textit{$\varphi-\psi$ compatible correspondence homomorphism} provided that
    \begin{enumerate}[label=(\alph*)]
        \item $\Phi(a \cdot x) = \varphi(a) \cdot \Phi(x)$ for all $a \in A$ and $x \in \sf{X}$.
        \item $\Phi(x \cdot b) = \Phi(x) \cdot \psi(b)$ for all $b \in B$ and $x \in \sf{X}$.
        \item $\psi(\la x,z \ra_B) = \la \Phi(x), \Phi(z) \ra_{M(D)}$ for all $x,z \in \sf{X}$.
    \end{enumerate}
    The maps $\varphi,\psi$ are called the \textit{coefficient maps} of $\Phi$.  Formally this is denoted $_{\varphi}\Phi_{\psi}: _{A}\sf{X}_B \longrightarrow\M(_{C}\sf{Y}_D)$.  We say that $\Phi$ is \textit{nondegenerate} if $\o\spn\{\Phi(\sf{X})D\} =\Phi(\sf{X}) \cdot D = \sf{Y}$ and both $\varphi,\psi$ are nondegenerate. 
\end{defn}

We refer the reader to \cite[Chapter~1]{enchilada} for a detailed treatment of correspondences (what they call right-Hilbert bimodules) along with their theory of balanced and exterior tensor products. 

\subsection{Hopf $\cst$-Algebras}

Hopf $\cst$-algebras were largely developed by Iorio, Vaes and Van Daele in \cite{iorio,hopf} to create an abstract framework for studying LCQGs and noncommutative duality.  Their definition (which differs throughout the literature) assumes less than what is imposed for a LCQG (see Definition \ref{def11}), so we begin here and gradually build up towards LCQGs through multiplicative unitaries via Kac systems.  Many of these concepts are pulled directly from \cite{fischer}; however, since Fischer's work isn't easily located in the literature, we transfer much of the exposition found there. 

\begin{defn} \label{def1}
    A \it{Hopf $\cst$-algebra} is a pair $(\h,\Delta)$ consisting of 
    \begin{enumerate}[label=(\alph*)]
        \item A $\cst$-algebra $\h$,
        \item A non-degenerate faithful $*$-homomorphism $\Delta: \h \longrightarrow \M(\h \otimes \h)$ called a \it{comultiplcation} such that \begin{align*} (\Delta \otimes \id_{\h}) \circ \Delta &= (\id_{\h} \otimes \Delta) \circ \Delta, \: \text{and} \\
            \o\spn\l\{\Delta(\h)(1_{\M(\h)} \otimes \h)\r\} &= \o\spn\l\{(1_{\M(\h)} \otimes \h)\Delta(\h)\r\} = \h \otimes \h.
        \end{align*}
    \end{enumerate}
    Moreover, we call a $*$-homomorphism $\epsilon: \h \longrightarrow \C$ satisfying \begin{align*}
        (\epsilon \otimes \id_{\h}) \circ \Delta = (\id_{\h} \otimes \epsilon) \circ \Delta = \id_{\h}
    \end{align*} a \it{counit}. 
\end{defn}

One also defines what is called an \it{antipode} for Hopf $\cst$-algebras, but we have no real use for it here.  We remark that many assume additional structure on Hopf $\cst$-algebras involving the Haagerup tensor product and no assumption on the faithfulness of $\Delta$ such as in \cite[Definition~2.4]{hopf}.  Regardless, there are three canonical examples of Hopf $\cst$-algebras.

\begin{ex} \label{ex1}
    Let $G$ be a locally compact group.  
    \begin{enumerate}
        \item $C_0(G)$ is a Hopf $\cst$-algebra with comultiplcation $\Delta: C_0(G) \longrightarrow \M(C_0(G) \otimes C_0(G)) \cong C_b(G \times G)$ given by $\Delta(f) = f \circ m$, where $m: G \times G \longrightarrow G$ is the group multiplication. 
        \item $\cst_r(G) \coloneq \lambda(\cst(G))$ is a Hopf $\cst$-algebra where $\lambda$ is the integrated form of the left regular representation: \begin{align*}
            \lambda(f) = \int_G f(s)\lambda(s) \: ds, \: f \in C_c(G).
        \end{align*} The comultiplcation is given by $\Delta \circ \lambda = \lambda \otimes \lambda$.
        \item $\cst(G)$ with comultiplcation $\delta_G$ is a Hopf $\cst$-algebra, which is the integrated form of the unitary representation $u \otimes u: s \mapsto u(s) \otimes u(s)$ where $u: G \longrightarrow \M(\cst(G))$ is the canonical map \cite[Remark~A.8]{enchilada}.
    \end{enumerate}
\end{ex}

Comultiplcations are examples of a broader concept known as \it{coactions}.  A suitable background for coactions of groups can be found in \cite[Section~A.3]{enchilada}.  Just as in the group case, coactions of Hopf $\cst$-algebras will be used extensively to define crossed products.

\begin{defn} \label{def2}
    A \it{coaction} of $(\h,\Delta)$ on a $\cst$-algebra $A$ is a $*$-homomorphism $\alpha: A \longrightarrow \M(A \otimes \h)$ such that 
    \begin{enumerate}[label=(\alph*)]
        \item $(\alpha \otimes \id_{\h}) \circ \alpha = (\id_A \otimes \Delta) \circ \alpha$ (the \it{coaction identity}), and
        \item $\alpha(A)(1_{\M(A)} \otimes \h) \subset A \otimes \h$.
    \end{enumerate}
    Moreover, we say that $\alpha$ is \it{continuous} provided that \begin{align*}
        \o\spn\l\{\alpha(A)(1_{\M(A)} \otimes \h)\r\} = \o\spn\l\{(1_{\M(A)} \otimes \h)\alpha(A)\r\} = A \otimes \h.
    \end{align*} 
\end{defn}

Throughout, we will always assume that our coactions are faithful and continuous unless specified otherwise. We use the term `continuous' to avoid re-using the term non-degenerate (as a coaction). Note that continuity of a coaction implies its range is contained in the relative $\h$-multiplier algebra $\M_{\h}(A \otimes \h)$, where \begin{align*}
    \M_{\h}(A \otimes \h) \coloneq \l\{m \in \M(A \otimes \h) \: | \: m(1_{\M(A)} \otimes \h) \cup (1_{\M(A) \otimes \h})m \subset A \otimes \h\r\}.
\end{align*}  For details of this construction, see \cite[Section~A.1]{enchilada}.  Similar to \cite{fischer}, we will call coactions of the Hopf $\cst$-algebra $\h$ on a $\cst$-algebra $A$ \it{$\h$-coactions}, and indicate this with the pair $(A,\alpha)$ when the choice of $\h$ is unambiguous.  Analogously to the classical theory of crossed product $\cst$-algebras, one must discuss the appropriate notion of covariance of representations of the systems $(A,\h,\alpha)$.  

\begin{defn} \label{def3}
    Let $(A,\alpha)$ be an $\h$-coaction and $B$ a $\cst$-algebra.  
    \begin{enumerate}[label=(\alph*)]
        \item A right \it{unitary corepresentation} of $(\h,\Delta)$ on $B$ is a unitary multiplier $u \in \M(B \otimes \h)$ such that $(\id_B \otimes \Delta)(u) = u_{12}u_{13}$.
        \item A left \it{unitary corepresentation} of $(\h,\Delta)$ on $B$ is a unitary multiplier $u \in \M(\h \otimes B)$ such that $(\Delta \otimes \id_B)(u) = u_{13}u_{23}$.
        \item A \it{(right) covariant homomorphism} of $(\h,\Delta)$ from $(A,\alpha)$ to $B$ is a pair $(\pi,u)$ where $\pi: A \longrightarrow \M(B)$ is a (possibly degenerate) representation and $u \in \M(B \otimes \h)$ is a (right) unitary corepresentation such that \begin{align*}
           (\pi \otimes \id_{\h}) \circ \alpha(\cdot) = \Ad(u) \circ (\pi(\cdot) \otimes 1_{\M\l(\h \r)}) .
        \end{align*}
        \item A \it{(left) covariant homomorphism} of $(\h,\Delta)$ from $(A,\alpha)$ to $B$ is a pair $(\pi,u)$ where $\pi: A \longrightarrow \M(B)$ is a (possibly degenerate) representation and $u \in \M(\h \otimes B)$ is a (left) unitary corepresentation such that \begin{align*}
           (\pi \otimes \id_{\h}) \circ \alpha(\cdot) = \Ad(\sigma(u)) \circ (\pi(\cdot) \otimes 1_{\M\l(\h\r)}) .
        \end{align*}
    \end{enumerate}
\end{defn}

Note the above definition still makes sense for degenerate homomorphisms $\pi$, as continuity of $\alpha$ implies $\alpha(A) \subset \M_{\h}(A \otimes \h)$ and $\pi \otimes \id_{\h}$ extends to the relative $\h$-multipliers by \cite[Proposition~A.6]{enchilada}.  That being said, we will assume our covariant homomorphisms are nondegenerate unless specified otherwise. Finally, we note that we will sometimes refer to right unitary corepresentations as simply `corepresentations' and will write `left' whenever dealing with left corepresentations.  We do the same when referring to covariance.

\begin{defn} \label{def4}
    Let $(A,\alpha)$ be an $\h$-coaction of $(\h,\Delta)$.
    \begin{enumerate}[label=(\alph*)]
        \item A \it{crossed product} is a triple $(A \rtimes_{\alpha} \what\h, j^{\alpha}_A, \mf{u}^{\alpha})$ consisting of a $\cst$-algebra $A \rtimes_{\alpha} \what\h$, a covariant homomorphism $(j^{\alpha}_{A},\mf{u}^{\alpha})$ where $j^{\alpha}_A: A \longrightarrow \M\l(A \rtimes_{\alpha} \what\h \r)$ and $\mf{u}^{\alpha} \in \M\l(A \rtimes_{\alpha} \what\h \otimes \h\r)$ satisfying the following universal property: For all covariant homomorphisms $(\pi,u): (A,\alpha) \longrightarrow \M(B)$ of $(A,\alpha)$ on $\cst$-algebras $B$, there exists a unique non-degenerate $*$-homomorphism $\pi \rtimes u: A \rtimes_{\alpha} \what\h \longrightarrow \M(B)$ such that \begin{align*}
            (\pi \rtimes u) \circ j^{\alpha}_A &= \pi, \: \text{and} \\ ((\pi \rtimes u) \otimes \id_{\h})(\mf{u}^{\alpha}) &= u.
        \end{align*}
        \item A \it{universal dual} of $(\h,\Delta)$ is a crossed product $(\what\h, \iota,\mf{u})$ for the trivial coaction $\C \longrightarrow \M(\C \otimes \h)$ by $z \mapsto z \otimes 1_{\M(\h)}$.  Note here that $\iota: \C \longrightarrow \M\l(\what\h \r)$ is just the embedding of scalars and so we usually exclude it in the notation.
        \item A \it{full universal dual} is a universal dual $\what\h$ of $\h$ such that there additionally exists a closed subspace $\mf{F} \subset \h^{*}$ such that
        \begin{enumerate}[label=(\roman*)]
            \item For all $\psi \in \mf{F}$ and $y \in \h$, the functionals $\psi \cdot y, y \cdot \psi \in \mf{F}$ where $\psi \cdot y: x \mapsto \psi(yx)$ and $y \cdot \psi: x \mapsto \psi(xy)$ ($\h$-invariance).
            \item The collection $(\id_{\what\h} \otimes \mf{F})(\mf{u}) \coloneq \l\{(\id_{\what\h} \otimes \psi)(\mf{u}) \: | \: \psi \in \mf{F}\r\} \subset \what\h$ is dense.   
        \end{enumerate}
    \end{enumerate} 
\end{defn}
  If $(\pi,u)$ is a (possibly degenerate) covariant homomorphism of $(A,\alpha)$ on a $\cst$-algebra $B$, we have that $(\iota,u)$ is covariant for the trivial coaction on $\C$ and thus there is a unique (non-degenerate) $*$-homomorphism $\mu_u \coloneq \iota \rtimes u: \what\h \longrightarrow \M(B)$ such that $(\mu_u \otimes \id_{\h})(\mf{u}) = u$.  Thus, covariant homomorphisms $(\pi,u)$ (degenerate or not) may be described as a pair (also called a covariant homomorphism) $(\pi,\mu_u): \l((A,\alpha),\what\h\r) \longrightarrow \M(B)$ such that \begin{align*}
    (\pi \otimes \id_{\h}) \circ \alpha = \Ad(\mu_u \otimes \id_{\h})(\mf{u})(\pi \otimes 1_{\M(\h)}).
\end{align*}  All LCQGs are Hopf $\cst$-algebras that (amongst other things) admit a universal dual in this sense, and so we will only concern ourselves with Hopf $\cst$-algebras which admit a universal dual. We also note the universal property of $\l(A \rtimes_{\alpha} \what\h, j^{\alpha}_A, j^A \r)$ states that there is a bijective correspondence between covariant homomorphisms $(\pi,\mu)$ of $(A,\alpha)$ and non-degenerate $*$-homomorphisms of $A \rtimes_{\alpha} \what\h$, where $j^A \coloneq \mu_{\mf{u}^{\alpha}}$.  Finally note that one can flip the unitary $\mf{u}^{\alpha}$ so that the crossed product $\l(A \rtimes_{\alpha} \what\h, j_A^{\alpha}, \sigma(\mf{u}^{\alpha}) \r)$ is universal for left covariant homomorphisms. 

\begin{ex} \label{ex3}
Let $G$ be a locally compact group with $\tilde\alpha: G \longrightarrow \text{Aut}(A)$ be a continuous action on a $\cst$-algebra $A$, and $\delta: A \longrightarrow \M(A \otimes \cst(G))$ a coaction of $G$ on $A$.
\begin{enumerate}[label=(\alph*)]
    \item  For $\h = C_0(G)$, observe that a right corepresentation $u \in \M(B \otimes C_0(G)) \cong C_b(G,M^{\beta}(B))$ is a strictly continuous unitary valued function satisfying $(\id_B \otimes \Delta)(u) = u_{12}u_{13}$.  In particular it's easy to see that \[A \rtimes_{\alpha} \cst(G) \cong A \rtimes_{\alpha} G.\]  Notice also that $\what \h = \cst(G)$ since unitary corepresentations of $C_0(G)$ correspond to unitary representations of $G$, which correspond to $*$-homomorphisms of $\cst(G)$ by the universal property.
    \item Now let $\h =\cst(G)$.  First, observe that the coaction $\delta: A \longrightarrow \M(A \otimes \cst(G))$ of $G$ on $A$ is exactly a coaction of the Hopf $\cst$-algebra $\cst(G)$ on $A$ by our definitions.  One also checks that \[A \rtimes_{\delta} C_0(G) \cong A \rtimes_{\delta} G\] canonically.  Dual to the situation in (a), we have $\what \h = C_0(G)$ since unitary corepresentations of $\cst(G)$ correspond to $*$-homomorphisms of $C_0(G)$ through the canonical unitary $w_G$, see \cite[Section~A.3]{enchilada}. 
\end{enumerate}
These two examples encapsulate the classical theory of crossed products. 
\end{ex}

We have defined the crossed product, but have yet to introduce a condition on the Hopf $\cst$-algebra to guarantee the existence of a crossed product in general.  It turns out that the existence of a full universal dual is a sufficient condition \cite[Theorem~1.10]{fischer}.  We are now in a position to discuss LCQGs and how they arise.

\begin{defn} \label{def5}
    Let $\H$ be a Hilbert space, $B$ a $\cst$-algebra, $\sf{E}_B$ a right Hilbert $B$-module and $V \in \B(\H \otimes \H)$ a unitary.
    \begin{enumerate}[label=(\alph*)]
        \item We say that $V$ is \it{multiplicative} provided that the following equation holds in $\H \otimes\H \otimes \H$: \begin{align*}
        V_{12}V_{13}V_{23} = V_{23}V_{12}.
    \end{align*} 
    \item A unitary $U \in \L_B(\sf{E} \otimes_{\ext} \H)$ is called a \it{right corepresentation of $V$} provided that $U_{12}U_{13}V_{23} = V_{23}U_{12}$ and we denote the subspace \begin{align*}
        \what{S}_U \coloneq \o\spn\l\{(\id_{\sf{E}} \otimes \omega)(U) \: | \: \omega \in \B(\H)_{*}\r\} \subset \L_B(\sf{E}).
    \end{align*}
    \item A unitary $W \in \L_B(\H \otimes_{\ext} \sf{E})$ is called a \it{left corepresentation of $V$} provided that $V_{12}W_{13}W_{23} = W_{23}V_{12}$ and we denote the subspace \begin{align*}
        S_W \coloneq \o\spn\l\{(\omega \otimes \id_{\sf{E}})(W) \: | \: \omega \in \B(\H)_{*}\r\} \subset \L_B(\sf{E}).
    \end{align*}
    \end{enumerate}   
\end{defn}

Note that a multiplicative unitary is clearly a right and left corepresentation of itself, and so we have the subspaces $S_V, \what{S}_V$ of $\B(\H)$.  By adding additional assumptions to the unitary $V$, we can ensure that $S_V, \what{S}_V$ are $\cst$-algebras with comultiplcations $\Delta,\what{\Delta}$ respectively.  

\begin{lemma}[{\cite[Theorem~3.8]{kac,fischer}}] \label{lem1}
    Let $\H$ be a Hilbert space, $B$ a $\cst$-algebra, $\sf{E}_B$ a right Hilbert $B$-module, $V \in \B(\H \otimes \H)$ a multiplicative unitary with $U \in \L_B(\sf{E} \otimes_{\ext} \H)$, $W \in \L_B(\H \otimes_{\ext} \sf{E})$ right and left corepresentations of $V$ respectively.  Moreover assume that $\what{S}_U,S_W$ are $\cst$-subalgebras of $\L_B(\sf{E})$ such that $U \in \M\l(\what{S}_U \otimes S_V \r)$ and $W \in \M\l(\what{S}_V \otimes S_W \r)$ for all such $\sf{E},U,W$.  Then there are comultiplcations $\Delta: S_V \longrightarrow \M\l(S_V \otimes S_V\r)$ and $\what\Delta: \what{S}_V \longrightarrow \M\l(\what{S}_V \otimes \what{S}_V\r)$ given by \begin{align*}
        \Delta(x) &\coloneq V(x \otimes 1_{\M(S_V)})V^{*}, \: \text{and} \\ \what\Delta(y) &\coloneq V^{*}\l(1_{\M\l(\what{S}_V\r)} \otimes y\r)V.
    \end{align*} 
    In particular, $(S_V,\Delta)$ and $\l(\what{S}_V, \what\Delta \r)$ are Hopf-$\cst$-algebras and in this case we denote $(S_r,\Delta) \coloneq (S_V,\Delta)$ and $\l(\what{S}_r,\what{\Delta} \r)  \coloneq \l(\what{S}_V,\what\Delta \r)$ respectively.
\end{lemma}
 Moreover, one can construct a full universal dual $\l(\l(\what{S}_f,\what{\Delta}_f\r), \mf{u}\r)$ of $(S_r,\Delta)$, namely by taking $\mf{F} = \l\{\omega|_{S_r} \: | \: \omega \in \B(\H)_{*}\r\}$.  Thus given a right corepresentation $U \in \M\l(\what{S}_U \otimes S_r\r)$ of $S_r$, the induced $*$-homomorphism $\mu_{U}: \what{S}_f \longrightarrow \what{S}_U$ surjects onto $\what{S}_U$.  In particular we have a surjective $*$-homomorphism $\what{\pi}_f \coloneq \mu_V: \what{S}_f \longrightarrow \what{S}_r$ called the \it{canonical surjection of $\what{S}_f$}.  Analogously there is a full universal dual $((S_f,\Delta_f),\what{\mf{u}})$ of $\l(\what{S}_r, \what\Delta\r)$ with \it{canonical surjection of $S_f$} denoted $\pi_f: S_f \longrightarrow S_r$.  Notably, what we call left corepresentations of $\what{S}_r$ are referred to simply as `corepresentations' in \cite{fischer}.  In particular, since we want $\what{\mf{u}}$ to implement a bijective correspondence between left corepresentations of $\what{S}_r$ (since these correspond to the left corepresentations of $V$) and non-degenerate $*$-homomorphisms of $S_f$, we take $\what{\mf{u}}$ to be the flip of the unitary one gets from Definition \ref{def4}.(b).  We now introduce what are known as \it{weak Kac systems with full duality}.  These are a special subclass of the more general \it{weak Kac systems}, see \cite[Section~2.2]{kk} for details.  

\begin{defn}[{\cite[Definition~2.1]{fischer}}] \label{def6}
    A \it{weak Kac system with full duality} is a triple $(\H,V,U)$ consisting of a Hilbert space $\H$, a multiplicative unitary $V \in \B(\H \otimes \H)$ and a self-adjoint unitary $U \in \B(\H)$ such that 
    \begin{enumerate}[label=(\alph*)]
        \item The hypothesis of Lemma \ref{lem1} holds for all $\cst$-algebras $B$, Hilbert $B$-modules $\sf{E}_B$ and left, right corepresentations of $V$ on $\sf{E}_B$.
        \item There exists a unitary $\mf{v} \in \M\l(\what{S}_f \otimes S_f \r)$ such that $\mf{v}$ is a right corepresentation of $S_f$ and a left corepresentation of $\what S_f$ where $\l(\l(\what{S}_f,\what{\Delta}_f \r),\mf{v}\r)$ and $\l((S_f,\Delta_f),\mf{v} \r)$ are universal duals of $(S_f,\Delta_f)$ and $\l(\what{S}_f,\what{\Delta}_f\r)$ respectively.
        \item The unitary $\mf{v}$ is such that $\l(\pi_f,\what{\pi}_f\r): \l((S_f,\Delta_f), \what{S}_f \r) \longrightarrow \B(\H)$ is covariant.
        \item The unitary $U$ is such that $\l(\what{\pi}_f, \Ad(U) \circ \pi_f\r): \l( \l(\what{S}_f, \what{\Delta}_f \r), S_f \r) \longrightarrow \B(\H)$ is covariant where the commutators $[\pi_f(x),U\pi_f(x')U^{*}]= [\what{\pi}_f\allowbreak{(y)},$ $ U\what{\pi}_f(y')U^{*} ] = 0$ for all $x,x' \in S_f$ and $y,y' \in \what{S}_f$.  
    \end{enumerate}
    Moreover we say that $(\H,V,U)$ is \it{strongly regular} provided that $\what{\pi}_f \rtimes \Ad(U) \circ \pi_f: \what{S}_f \rtimes_{\what{\Delta}_f} S_f \cong \K$.  For weak Kac systems with full duality, we call 
    \begin{enumerate}[label=(\alph*)]
        \item $(S_r,\Delta)$ the \it{quantum group associated to $V$},
        \item $\l(\what{S}_r,\what\Delta \r)$ the \it{dual quantum group associated to $V$},
        \item $(S_f,\Delta_f)$ the \it{universal quantum group associated to $V$}, and
        \item $\l(\what{S}_f,\what{\Delta}_f\r)$ the \it{universal dual quantum group associated to $V$}.
    \end{enumerate}
\end{defn}

Note by `universal dual' in condition (b), we mean in the sense that left corepresentations $u$ of $\what S_f$ are in bijection with non-degenerate $*$-ho\allowbreak{momorphisms} of $S_f$ via $\l(\id_{\what S_f} \otimes \mu_u \r)(\mf{v}) = u$, and right corepresentations $w$ of $S_f$ are in bijection with non-degenerate $*$-homomorphisms of $\what S_f$ via $\l(\mu_w \otimes \id_{S_f} \r)(\mf{v}) = w$. 

\subsection{\centering{$\cst$-algebraic Locally Compact Quantum Groups}}  We define in the sense of Kustermans and Vaes \cite{lcqg} what we mean by LCQG. 

\begin{defn} \label{def11}
    Let $(\G,\Delta)$ be a Hopf $\cst$-algebra with faithful KMS-weights $\varphi,\psi$ respectively satisfying
    \begin{enumerate}[label=(\alph*)]
        \item $\varphi\l((\omega \otimes \id_{\G}) \circ \Delta(x) \r) = \omega\l(1_{M(\G)} \r)\varphi(x)$ for all $x \in \mathcal{M}_{\varphi}^{+}$ and $\omega \in \G^{*}_{+}$ (\it{left invariance}), and
        \item $\psi\l((\id_{\G} \otimes \omega) \circ \Delta(x) \r) = \omega\l(1_{M(\G)} \r)\psi(x)$ for all $x \in \mathcal{M}_{\varphi}^{+}$ and $\omega \in \G^{*}_{+}$ (\it{right invariance}).
    \end{enumerate}  Then $\G$ is called a ($\cst$-algebraic reduced) \it{locally compact quantum group} (LCQG) and $\varphi,\psi$ are called the left and right \it{Haar weights} respectively.  
\end{defn}

For a detailed background on these objects, see \cite[Section~2]{buss,weight}.  We denote $L^2(\G)$ to be the Hilbert space associated to the left Haar weight $\varphi$, and regard $\G \subset \B(L^2(G)) \cong\M(\K)$ as a non-degenerate $\cst$-subalgebra of $\B(L^2(\G)) \cong\M(\K)$ whenever convenient.  Importantly, this construction is (one way to see) how LCQGs can be realized as Hopf $\cst$-algebras arising from multiplicative unitaries.  Associated to a KMS weight $\varphi$ is a unique anti-unitary operator $J$ called the \it{modular conjugation of $\varphi$}.  A very detailed discussion of this operator can be found in \cite{stratila}. 

\begin{lemma}[{\cite[Section~2]{fischer}}]
    Let $(\G,\Delta)$ be a LCQG.  Then $(L^2(\G),W,U)$ is a weak Kac system with full duality, where $W$ is the left corepresentation and $U = \what J J$ for $J,\what J$ the modular conjugations associated to the left Haar weight of $\G$ and its dual $\what \G$ respectively.
\end{lemma}

Thus given $(\G,\Delta)$ in the sense of Kustermans and Vaes, one has the weak Kac system with full duality $(L^2(\G),W,U)$. One then induces the quantum groups $(S_r,\Delta), (S_f,\Delta_f),(\what S_r,\what\Delta)$, and $(\what S_f,\what\Delta_f)$.  In the Kusterman Vaes conventions, one also obtains similar objects $(\G,\Delta),(\G_u,\Delta_u),(\what \G,\what\Delta)$, and $(\what \G_u, \what\Delta_u)$, see \cite{kust}.  The correspondence between these is the following \begin{align*}
    (S_r,\Delta) &= (\what \G, \what \Delta^{\op}), \\ (S_f,\Delta_f) &= (\what \G_u, \what \Delta_u^{\op}), \\ (\what S_r, \what\Delta) &= (\G,\Delta), \\ (\what S_f, \what \Delta_f) &= (\G_u,\Delta_u)
\end{align*} where $^{\op}$ denotes the flip of the comultiplcations. For more details on how to swap between these conventions, see \cite[Section~1.3.5]{fischer04}.  Since we wish to talk about crossed product duality for LCQGs, we mention what regularity is.

\begin{defn} \label{def12}
    Let $(\H,V,U)$ be a weak Kac system with full duality.  We say the Kac system $(\H,V,U)$ is \it{regular} provided that \[\mathcal{C}(V) \coloneq \o\spn\l\{(\id \otimes \omega)(\Sigma V) \: | \: \omega \in \B(\H)_{*}\r\} = \K\]
\end{defn}

Note that when we say `$(\H,V,U)$ is a weak Kac system with full duality arising from a regular LCQG', we will always mean that given a LCQG $\G$, we take $(\H,V,U) = (L^2(\G),W,U)$ for $W,U$ defined above such that Definition \ref{def12}(a) holds.  Moreover, by \cite[Remark~2.2]{fischer} regularity is equivalent to $\o\spn\l\{\what S_r \cdot \Ad(U) \circ S_r\r\} = \K$.
\subsection{More Crossed Product Constructions}

We are finally in a position to discuss crossed products in a detailed manner.  For this, fix a weak Kac system with full duality $(\H,V,U)$ arising from a regular LCQG $\G$. We call a coaction $\alpha: A \longrightarrow \M(A \otimes S_f)$ of $S_f$ a \it{full $S$-coaction} and a coaction $\beta: B \longrightarrow \M\l(B \otimes \what S_f\r)$ of $\what{S}_f$ a \it{full $\what{S}$-coaction}.  We will always assume our coactions are continuous. Crossed products by a full $S$-coaction $(A,\alpha)$ will be denoted $A \rtimes_{\alpha} \what{S}_f$ and similarly for full $\what{S}$-coactions $(B,\beta)$ by $B \rtimes_{\beta} S_f$.  This reflects the fact that covariant homomorphisms of $((A,\alpha),S_f)$ are written as pairs $(\pi,\mu_u): \l(A,\what{S}_f \r) \longrightarrow \M(D)$ where $u \in \M(D \otimes S_f)$ is the associated right corepresentation of $S_f$ on $D$.  Analogously for systems $\l((B,\beta), \what{S}_f\r)$ we have pairs $(\pi,\mu_u): (B,S_f) \longrightarrow \M(D)$ where $u \in \M\l(\what S_f \otimes D\r)$ is the associated left corepresentation.  In particular, covariance of homomorphisms for full $S$-coactions $(\pi,\mu): \l((A,\alpha),\what S_f \r) \longrightarrow \M(D)$ is then given as \begin{align*}
    (\pi \otimes \id_{S_f}) \circ \alpha(\cdot) = \Ad(\mu \otimes \id_{S_f})
(\mf{v})(\pi(\cdot) \otimes 1_{\M(S_f)}).\end{align*}  Similarly, covariance of homomorphisms $(\nu,\eta): \l((B,\beta),S_f \r) \longrightarrow \M(D)$ for full $\what S$-coactions is then \begin{align*} 
       \l(\nu \otimes \id_{\what S_f} \r) \circ \beta(\cdot) = \Ad\l(\eta \otimes \id_{\what S_f} \r)
(\sigma(\mf{v}))\l(\nu(\cdot) \otimes 1_{\M( \what S_f)} \r).
\end{align*}  This is the case since for these Kac systems since we know $S_f$ and $\what S_f$ admit the unique universal duals $\l(\what S_f, \mf{v}\r)$ and $(S_f,\mf{v})$ respectively.  In particular, See the comments proceeding Definition \ref{def6}.

\begin{defn} \label{def13}
    Let $(A,\alpha)$ be a full $S$-coaction and $(B,\beta)$ a full $\what{S}$-coaction. 
    \begin{enumerate}
        \item Define the maps $j^{\alpha}_{A,r} \coloneq (\id_A \otimes \pi_f) \circ \alpha$ and $j^A_r \coloneq 1_{M(A)} \otimes \what{\pi}_f$.  We call the map $\Lambda_{\alpha} \coloneq j^{\alpha}_{A,r} \rtimes j^A_r: A \rtimes_{\alpha} \what{S}_f \longrightarrow\M(A \otimes \K)$ the \it{$S$-regular representation}.  We define the \it{$S$-reduced crossed product} to be $A \rtimes_{\alpha,r} \what{S}_r \coloneq \Lambda_{\alpha}\l(A \rtimes_{\alpha} \what{S}_f \r)$
        \item Define the maps $\j^{\:\beta}_{B,r} \coloneq (\id_B \otimes \Ad(U) \circ \what{\pi}_f) \circ \beta$ and $\j^{\:B}_{r} \coloneq 1_{M(B)} \otimes \pi_f$.  We call the map $\what{\Lambda}_{\beta} \coloneq \j^{\:\beta}_{B,r} \rtimes \j^{\:B}_r: B \rtimes_{\beta} S_f \longrightarrow\M(B \otimes \K)$ the \it{$\what{S}$-regular representation}.  We define the \it{$\what{S}$-reduced crossed product} to be $B \rtimes_{\beta,r} S_r \coloneq \what{\Lambda}_{\beta}\l(B \rtimes_{\beta} S_f \r)$.
    \end{enumerate}
\end{defn}

Note that the above homomorphism pairs are covariant by \cite[Lemma~4.5,Remark~4.13]{fischer}.  We will also refer to these objects as simply `regular representations' and `reduced crossed products' when no confusion arises.  

\begin{rem}
    Although we only state our results for full $S$-coactions, we remark that every result proven also holds for full $\what S$-coactions via switching from right to left corepresentations and using the formulas defined for full $\what S$-coactions.  Moreover, there are many statements that obviously are true for full $\what S$-coactions despite their absence. We opt to include the details for the crossed product, dual coactions, the regular representation, and the canonical surjection for full $\what S$-coactions since once one has the necessary formulas from these statements, the proof techniques employed for full $S$-coactions carry over verbatim for full $\what S$-coactions. Examples of statements that have obvious carry-over to full $\what S$-coaction include the theory of equivariant morphisms, crossed product duality, the propositions in section $4$, and the statement of the main result.  Discussion about full $\what S$-coactions can be found throughout \cite{fischer,fischer04}.
\end{rem}

As with classical crossed products, there are accompanying dual coactions associated to crossed product $\cst$-algebras by quantum groups.  In particular if $(A,\alpha)$ is a full $S$-coaction, then there are naturally defined full $\what{S}$-coactions $\l(A \rtimes_{\alpha} \what{S}_f, \what{\alpha} \r)$  and $\l(A \rtimes_{\alpha,r} \what{S}_r, \what{\alpha}^n\r)$.

\begin{lemma} \label{lem5}
    Let $(\h,\Delta)$ be a Hopf $\cst$-algebra that admits a full universal dual $\l(\what\h ,\what\Delta\r)$, and suppose that $(A,\alpha)$ is an $\h$-coaction.  Then $\l(j^{\alpha}_A \otimes 1_{M\l(\what\h\r)}, \l(j^A \r.\r.$ $\l.\l. \otimes \id_{\what\h} \r) \circ \what\Delta\r): \l(A, \what\h \r) \longrightarrow\M\l(\l(A \rtimes_{\alpha} \what\h \r) \otimes \what\h\r)$ is a covariant homomorphism whose integrated form $\what{\alpha}$ we call the \it{dual coaction of $\alpha$}.
\end{lemma}

\begin{lemma} \label{lem6}
    Let $(A,\alpha)$ and $(B,\beta)$ be full $S$ and $\what{S}$-coactions respectively.  Then
    \begin{enumerate}[label=(\alph*)]
        \item There is a full $\what{S}$-coaction $\l(A \rtimes_{\alpha} \what{S}_f, \what{\alpha} \r)$ given by \[\what{\alpha} \coloneq \left(j^{\alpha}_A \otimes 1_{M\left(\what{S}_f\right)}\right) \rtimes \left(j^A \otimes \id_{\what{S}_f}\right) \circ \what{\Delta}_f: A \rtimes_{\alpha} \what{S}_f \longrightarrow\M\l(A \rtimes_{\alpha} \what{S}_f \otimes \what{S}_f\r).\]
        \item There is a full $\what{S}$-coaction $\l(A \rtimes_{\alpha,r} \what{S}_r, \what{\alpha}^n \r)$ making the following diagram commute: \[\begin{tikzcd}
A \rtimes_{\alpha} \what{S}_f \arrow[rrdd, "\Lambda_{\alpha}"'] \arrow[rrrr, "\left(\Lambda_{\alpha} \otimes \id_{\what{S}_f} \right) \circ \widehat{\alpha}"] &  &                                                                       &  & {M\left(A \rtimes_{\alpha,r} \widehat{S}_r \otimes \what{S}_f\right)} \\
                                                                                                                                                    &  &                                                                       &  &                                                                                \\
                                                                                                                                                    &  & {A \rtimes_{\alpha,r} \widehat{S}_r} \arrow[rruu, "\widehat{\alpha}^n"'] &  &                                                                               
\end{tikzcd}\]
        \item There is a full $S$-coaction $\l(B \rtimes_{\beta} S_f, \what{\beta} \r)$ given by \[\what{\beta} \coloneq \left(\j^{\: \beta}_B \otimes 1_{M\left(S_f\right)}\right) \rtimes \left(\j^{\: B} \otimes \id_{S_f}\right) \circ \Delta_f: B \rtimes_{\beta} S_f \longrightarrow\M\l(B \rtimes_{\beta} S_f \otimes S_f\r).\]
        \item There is a full $S$-coaction $\l(B \rtimes_{\beta,r} S_r, \what{\beta}^n \r)$ making the same diagram as in (b) commute.
    \end{enumerate}
\end{lemma}

\begin{proof}
    Parts \it{(a)} and \it{(c)} are applying Lemma \ref{lem5}.  Parts \it{(b)} and \it{(d)} are \cite[Proposition~4.9,Lemma~4.11]{fischer} and \cite[Remark~4.13]{fischer}.
\end{proof}

\begin{defn} \label{def14}
Let $(A,\alpha),(B,\beta)$ be full $S$-coactions $\alpha: A \longrightarrow\M(A \otimes S_f)$, $\beta: B \longrightarrow\M(B \otimes S_f)$.  A (possibly degenerate) $*$-homomorphism $\varphi: A \longrightarrow B$ is called \textit{$\alpha-\beta$ equivariant} provided the following diagram commutes: \[\begin{tikzcd}
A \arrow[dd, "\varphi"'] \arrow[rrr, "\alpha"] &  &  & M_{S_f}(A \otimes S_f) \arrow[dd, "\overline{\varphi \otimes \id_{S_f}}"] \\
                                               &  &  &                                                                                       \\
B \arrow[rrr, "\beta"']                        &  &  & M_{S_f}(B \otimes S_f)                                                    
\end{tikzcd}.\]
\end{defn}

Note here that although $\varphi$ is possibly degenerate, $\varphi \otimes \id$ still extends uniquely to a $S_f$-strict to $S_f$-strictly continuous $*$-homomorphism $\overline{\varphi \otimes \id}$ on the relative $S_f$-multipliers of $A \otimes S_f$.  This is (as previously referenced) a consequence of \cite[Proposition~A.6]{enchilada}.  Equivariant homomorphisms are extremely important, since equivariance is carried to the dual coaction through taking the \it{$S$-crossed product} $\varphi \rtimes \what{S}_f$ of the homomorphism $\varphi$.   We refer the reader to the appendix of \cite{enchilada} for a detailed discussion of equivariant morphisms.

\begin{lemma}
    \label{lem8} The $S$-crossed product $\varphi \rtimes \what{S}_f$ is $\widehat{\alpha}-\widehat{\beta}$ equivariant.  
\end{lemma}

\section{Duality and Ideals of Crossed Products}
This section deals with the bulk of the technical results needed to prove Theorem \ref{thm3}.  Throughout this section, we again fix a weak Kac system with full duality $(\H,V,U)$ arising from a regular LCQG $\G$. We now mention the two types of coactions we will consider throughout the rest of this work. With this, we can state the duality theorems.  As before, several of these results we directly credit to \cite{fischer}, and state them here for convenience's sake.  Let $(A,\alpha)$ and $(B,\beta)$ be a full $S$ and $\what S$-coaction respectively.  Recall that a strict cocycle (for, say the $S$-coaction $\alpha$) is a unitary $u \in \M(A \otimes S_f)$ such that  
\begin{enumerate}[label=(\alph*)]
    \item $(\id_A \otimes\alpha)(u) = u_{12}(\alpha \otimes \id_{S_f})(u)$ and
    \item $\o\spn\l\{(1_A \otimes S_f)\Ad(u) \circ \alpha(A)\r\} = A \otimes S_f$.
\end{enumerate} The maps $\Ad(U) \circ \what \pi_f: \what S_f \longrightarrow \B(\H)$ and $\Ad(U) \circ \pi_f: S_f \longrightarrow \B(\H)$ correspond to unitary corepresentations \[\mf{w} \coloneq (\Ad(U) \circ \what \pi_f \otimes \id_{S_f})(\mf{v}) \in \M(\K \otimes S_f)\] and \[\what{\mf{w}} \coloneq (\Ad(U) \circ \pi_f \otimes \id_{\what S_f})(\sigma(\mf{v})) \in \M(\K \otimes \what S_f)\] respectively by Definition \ref{def6}. One can check by \cite[Example~1.3]{fischer} that $((\id_A \otimes \Sigma) \circ (\alpha \otimes \id_{\K}), A \otimes \K)$ and $((\id_B \otimes \what\Sigma) \circ (\beta \otimes \id_{\K}), B \otimes \K)$ are full $S$ and $\what S$-coactions respectively, where $\Sigma: S_f \otimes \K \longrightarrow \K \otimes S_f$ and $\what\Sigma: \what S_f \otimes \K \longrightarrow \K \otimes \what S_f$ are the flip isomorphisms. One can then verify that $\mf{w}_{23}, \what{\mf{w}}_{23}$ are strict cocycles for $(\id_A \otimes \Sigma) \circ (\alpha \otimes \id_{\K})$ and $(\id_B \otimes \what\Sigma) \circ (\beta \otimes \id_{\K})$ respectively, so that one obtains perturbed full $S$ and $\what S$-coactions on $A \otimes \K$ and $B \otimes \K$ respectively via \begin{align*}
    \epsilon &\coloneq \Ad(\mf{w}_{23}) \circ (\id_A \otimes \Sigma) \circ (\alpha \otimes \id_{\K}), \:\text{and} \\ \what{\epsilon} &\coloneq \Ad(\what{\mf{w}}_{23}) \circ (\id_B \otimes \what\Sigma) \circ (\beta \otimes \id_{\K})
\end{align*} by \cite[Example~1.15]{fischer}.  We will denote the coaction $(\id_A \otimes \Sigma) \circ (\alpha \otimes \id_{\K})$ by $(\alpha \otimes_{*} \id_{\K})$ as is common practice.  Clearly, we have defined the obvious generalization of the usual coaction on the stabilization of an algebra in the statement of Katayama duality, see \cite[Remark~5]{nil}.

\begin{lemma} \label{lem9}
    Let $(A,\alpha)$ and $(B,\beta)$ be full $S$ and $\what S$-coactions respectively.  
    \begin{enumerate}[label=(\alph*)]
        \item The pair $(\Lambda_{\alpha}, 1_{M(A)} \otimes \Ad(U) \circ \pi_f): \l(A \rtimes_{\alpha} \what S_f, \what{\alpha}\r) \longrightarrow \M(A \otimes \K)$ is covariant and induces an $\what{\what{\alpha}}-\epsilon$ equivariant surjective homomorphism $\Phi_{\alpha}: A \rtimes_{\alpha} \what S_f \rtimes_{\what{\alpha}} S_f \longrightarrow A \otimes \K$.
        \item The pair $(\what{\Lambda}_{\beta}, 1_{M(B)} \otimes \what{\pi}_f): \l(B \rtimes_{\alpha} S_f, \what{\beta}\r) \longrightarrow \M(B \otimes \K)$ is covariant and induces an $\what{\what{\beta}}-\what{\epsilon}$ equivariant surjective homomorphism $\what{\Phi}_{\beta}: B \rtimes_{\beta} S_f \rtimes_{\what{\beta}} \what{S}_f \longrightarrow B \otimes \K$.
    \end{enumerate}
    We call the maps $\Phi_{\alpha}$ and $\what{\Phi}_{\beta}$ the \it{canonical surjections}.
\end{lemma}

These are the obvious generalizations of Nilsen's canonical surjections from \cite{nil}.

\begin{defn} \label{def15}
    Let $(A,\alpha)$ be a full $S$-coaction.  
    \begin{enumerate}[label=(\alph*)]
        \item $\alpha$ is called \it{maximal} provided that $\Phi_{\alpha}$ is an isomorphism.  
        \item $\alpha$ is called \it{normal} provided that $j_A^{\alpha}$ is faithful.
    \end{enumerate}
      A remarkable fact is that dual coactions on full crossed products are always maximal, see \cite[Proposition~6.6]{fischer}.  As a remark, the dual coaction on the reduced crossed product $\what{\alpha}^n$ is always a \it{normalization} of $\what{\alpha}$, see \cite[Lemma~4.11,Remark~4.13]{fischer}.  This is why we use a $^n$ to denote this coaction.  
\end{defn}

We now state crossed product duality for full $S$-coactions, and explain where this statement comes from.
 
\begin{theorem}[{\cite{fischer}}] 
\label{thm2}
    Let $(A,\alpha)$ be a full $S$-coaction.  
    \begin{itemize}
        \item $(\textit{Maximal Duality})$ If $\alpha$ is maximal, there are $\widehat{\widehat{\alpha}}-\epsilon$ and $\widehat{\hat{\alpha}}^n-\epsilon$ equivariant isomorphisms: \begin{align*}
            &\left(A \rtimes_{\alpha} \widehat{S}_f \rtimes_{\widehat{\alpha}} S_f, \widehat{\widehat{\alpha}}\right) \cong \left(A \otimes \K, \epsilon \right), \: \: \text{and} \\ & \left(A \rtimes_{\alpha,r} \widehat{S}_r \rtimes_{\widehat{\alpha}^n} S_f, \widehat{\hat{\alpha}}^n \right) \cong \left(A \otimes \K, \epsilon \right).
        \end{align*} for $\widehat{\hat{\alpha}}^n$ the dual coaction of $\widehat{\alpha}^n$ and $\epsilon$ given above.
        \item $(\textit{Normal Duality})$  If $\alpha$ is normal, there is a $\left(\widehat{\widehat{\alpha}}\right)^n-\epsilon$ equivariant isomorphism: \begin{align*}
            &\left(A \rtimes_{\alpha} \what{S}_f \rtimes_{\widehat{\alpha},r} S_r, \left(\widehat{\widehat{\alpha}}\right)^n \right) \cong \left(A \otimes \K, \epsilon \right).
            \end{align*}
    \end{itemize}
\end{theorem}
The $\widehat{\hat{\alpha}}^n-\epsilon$ equivariant isomorphism follows from the regular representation $\Lambda_{\alpha}: \left(A \rtimes_{\alpha} \widehat{S}_f, \widehat{\alpha}\right) \longrightarrow \left(A \rtimes_{\alpha,r} \widehat{S}_r, \widehat{\alpha}^n\right)$ being a normalization of the dual coaction $\widehat{\alpha}$ see \cite[Lemma~4.11]{fischer}.  In particular, this means $\Lambda$ is an $\widehat{\alpha}-\widehat{\alpha}^n$ equivariant surjection, such that its crossed product \[\Lambda \rtimes S_f: A \rtimes_{\alpha} \widehat{S}_f \rtimes_{\widehat{\alpha}} S_f \longrightarrow A \rtimes_{\alpha,r} \widehat{S}_r \rtimes_{\widehat{\alpha}^n} S_f\] is an isomorphism, which is $\what{\what{\alpha}}-\what{\hat{\alpha}}^n$ equivariant by Lemma \ref{lem8}. We now possess the tools to productively discuss coaction invariance of ideals, and how to prove the main result.  Note if $I \lhd A$ is a closed two-sided ideal of $A$, then we may (and will implicitly throughout) identify $\M_{S_f}(I \otimes S_f) \subset \M_{S_f}(A \otimes S_f) \subset \M(A \otimes S_f)$ via \cite[Lemma~1.46]{enchilada}.

\begin{defn}
\label{def17}
    Let $(A,\alpha)$ be a full $S$-coaction and $I \lhd A$ a (closed two-sided) ideal of $A$.  We say that $I$ is \textit{$\alpha$-invariant} provided that $\alpha(I) \subset \M_{S_f}(I \otimes S_f)$, and the restriction map $\alpha|_{I}:I \longrightarrow \M(I \otimes S_f)$ defines a full $S$-coaction on $I$.  Moreover, we denote the collection of all $\alpha$-invariant ideals of $A$ by $\scr{I}_{\alpha}(A)$. 
\end{defn}

Upon restriction to invariant ideals, we note that continuity of the coaction is preserved. 

\begin{lemma} \label{lem10}
     Let $(A,\alpha)$ be a full $S$-coaction and $I \lhd A$.  Then $I \in \scr{I}_{\alpha}(A)$ if and only if \[\o\spn\left\{\alpha(I)(1_{M(A)} \otimes S_f)\right\} = I \otimes S_f.\] 
\end{lemma}

\begin{proof}
    This is the proof of \cite[Proposition~2.6.23]{buss} included for completeness.  If $I \in \scr{I}_{\alpha}(A)$, then $\alpha(I)(1_{M(I)} \otimes S_f) \subset I \otimes S_f$. Now let $x \in I \otimes S_f$.  Since $\alpha$ is continuous we have that $x \approx \sum\limits_{k=1}^n \alpha(a_k)(1_{M(A)} \otimes z_k)$ for $a_k \in A$ and $z_k \in S_f$.  However since $\alpha|_{I}$ is a coaction, it is non-degenerate so that for $(e_i)_{i} \subset I$ an approximate unit, we have $\alpha(e_i) \longrightarrow 1_{A \otimes S_f}$ strictly, implying \begin{align*}
        x &\approx \alpha(e_i)x \\ &\approx \alpha(e_i)\l(\sum\limits_{k=1}^n \alpha(a_k)(1_{M(A)} \otimes z_k)\r) \\ &= \sum\limits_{k=1}^n \alpha(e_ia_k)(1_{M(A)} \otimes z_k) \\ &= \sum\limits_{k=1}^n \alpha(e_ia_k)(1_{M(I)} \otimes z_k) \\ & \in \o\spn\l\{\alpha(I)(1_{M(I)} \otimes S_f)\r\}
    \end{align*} proving continuity of $\alpha|_{I}$.  Conversely if the latter condition holds, then $\alpha|_{I}(I) \subset M_{S_f}(I \otimes S_f)$, where $\alpha|_{I}$ is a faithful non-degenerate $*$-homomorphism satisfying the coaction identity.  Hence, $I \in \scr{I}_{\alpha}(A)$. 
\end{proof}

To prove the main result, we collect several technical propositions.  Thankfully, these are roughly the steps needed to prove \cite[Theorem~3.2]{ladder}, just in higher generality.  This is thanks to the employment of the ladder technique. To begin, we define correspondence coactions on quantum groups in line with \cite[Definition~2.10]{enchilada}.

\begin{defn}
\label{def18}
     Let $\sf{X}$ be an $(A,B)$ correspondence with full $S$-coactions $(A,\alpha)$ and $(B,\beta)$.  View $S_f$ as an $(S_f,S_f)$-correspondence and form the exterior tensor product to get: \[_{A \otimes S_f}(\sf{X} \otimes_{\ext} S_f)_{B \otimes S_f}\] which is an $(A \otimes S_f,B \otimes S_f)$-correspondence.  A \textit{correspondence coaction compatible with $\alpha$ and $\beta$} is a non-degenerate correspondence homomorphism $\zeta: \sf{X} \longrightarrow \M(_{A \otimes S_f}(\sf{X} \otimes_{\ext} S_f)_{B \otimes S_f})$ such that
\[(\zeta \otimes \id_{S_f}) \circ \zeta = (\id_{\sf{X}} \otimes \Delta_f) \circ \zeta.\]
    Moreover, we say that $\zeta$ is \textit{continuous} provided that \[\o\spn\{(1_{M(A)} \otimes S_f)\zeta(\sf{X})\} = \: _{A \otimes S_f}(\sf{X} \otimes_{\ext} S_f)_{B \otimes S_f}.\] 
\end{defn}

\begin{prop}
    \label{prop1} Let $(A,\alpha)$, $(B,\beta)$ be full $S$-coactions and $\sf{X}$ be an $(A,B)$-imprimitivity bimodule, with $J \lhd B$, and $I \lhd A$ the ideal of $A$ associated to $J$ under the Rieffel correspondence (that is, $I = \sf{X}-\text{Ind}^A_B(J)$).  Moreover let $\zeta$ be a continuous $\alpha-\beta$ compatible coaction on $\sf{X}$.  Then $I$ is $\alpha$-invariant if and only if $J$ is $\beta$-invariant.
\end{prop}

\begin{proof}
    Suppose that $J \lhd B$ is $\beta$-invariant. It suffices to prove that $I$ is $\alpha$-invariant by symmetry.  To do this, it suffices to appeal to Lemma \ref{lem10}.  Computing we have that \begin{align*}  \o\spn\{(1 \otimes S_f)\alpha(I)\} &=\o\spn\left\{(1 \otimes S_f) \alpha\left(_{A}\la \sf{X} \cdot J, \sf{X} \ra \right) \right\} \\&= \o\spn\left\{(1 \otimes S_f) \: _{\M}\la \zeta(\sf{X} \cdot J), \zeta(\sf{X}) \ra \right\} \\&= \o\spn\left\{ \: _{\M}\la (1 \otimes S_f)\zeta(\sf{X} \cdot J), \zeta(\sf{X}) \ra \right\} \\&= \o\spn\left\{ \: _{\M}\la (1 \otimes S_f)^{*}(1 \otimes S_f)\zeta(\sf{X} \cdot J), \zeta(\sf{X}) \ra \right\} \\& = \o\spn\left\{ \: _{\M}\la (1 \otimes S_f)\zeta(\sf{X} \cdot J), (1 \otimes S_f)\zeta(\sf{X}) \ra \right\} \\& = \o\spn\left\{ \: _{\M}\la (1 \otimes S_f)\zeta(\sf{X}) \beta(J), (1 \otimes S_f)\zeta(\sf{X}) \ra \right\} \\& = \o\spn\left\{ \: _{\M}\la (\sf{X} \otimes_{\ext} S_f ) \beta(J), \sf{X} \otimes_{\ext} S_f  \ra \right\} \\& = \o\spn\left\{ \: _{\M}\la (\sf{X} \otimes_{\ext} S_f )(1 \otimes S_f) \beta(J), \sf{X} \otimes_{\ext} S_f  \ra \right\} \\& = \o\spn\left\{ \: _{\M}\la (\sf{X} \otimes_{\ext} S_f )(J \otimes S_f ), \sf{X} \otimes_{\ext} S_f  \ra \right\} \\& = \o\spn\left\{ \: _{\M}\la \sf{X} \cdot J \otimes_{\ext} S_f, \sf{X} \otimes_{\ext} S_f  \ra \right\}\end{align*} where $\M = \M(A \otimes S_f)$. Note that the $M(A \otimes S_f)$-valued inner product extends that of the $A \otimes S_f$-valued inner product.  Hence we have

    \begin{align*} \o\spn\left\{ \: _{\M}\la \sf{X} \cdot J \otimes_{\ext} S_f, \sf{X} \otimes_{\ext} S_f  \ra \right\} &= \o\spn\left\{ \: _{A \otimes S_f}\la \sf{X} \cdot J \otimes_{\ext} S_f , \sf{X} \otimes_{\ext} S_f  \ra \right\} \\& = \o\spn\left\{ \: _{A} \la \sf{X} \cdot J, \sf{X} \ra \otimes \: _{S_f}\la S_f, S_f \ra\right\} \\& = I \otimes S_f.\end{align*} Hence, we have shown \begin{align*}
        \o\spn\l\{(1 \otimes S_f)\alpha(I)\r\} &= I \otimes S_f \\ &= (I \otimes S_f)^{*} \\ &= \o\spn\l\{\alpha(I)(1 \otimes S_f)\r\}.
    \end{align*} Thus, $J$ being $\beta$-invariant implies $I$ is $\alpha$-invariant by Lemma \ref{lem10}.  
\end{proof}

\begin{prop}
    \label{prop2} Let $(A,\alpha)$ be a full $S$-coaction with $I \in \scr{I}_{\alpha}(A)$.  Then the inclusion map $\varphi: I \hookrightarrow A$ is $\alpha|-\alpha$ equivariant and the crossed product $\varphi \rtimes \widehat{S}_f$ is an isomorphism of $I \rtimes_{\alpha|} \widehat{S}_f$ onto an $\widehat{\alpha}$-invariant ideal $I \rtimes_{\alpha} \widehat{S}_f \in \scr{I}_{\widehat{\alpha}}\left(A \rtimes_{\alpha} \widehat{S}_f \right)$.  Moreover, the image $I \rtimes_{\alpha,r} \widehat{S}_r \coloneq \Lambda_{\alpha} \left(I \rtimes_{\alpha} \widehat{S}_f \right)$ under the regular representation $\Lambda_{\alpha}$ is $\widehat{\alpha}^n$-invariant, where $\widehat{\alpha}^n$ is the dual coaction of $\alpha$ on the reduced crossed product $A \rtimes_{\alpha,r} \widehat{S}_r$.  
\end{prop}

\begin{proof}
    It's immediately clear that $\varphi$ is $\alpha|-\alpha$ equivariant.  A standard argument shows that $j_{A}|I \rtimes j^A = \varphi \rtimes \what{S}_f$ is faithful, see \cite[Proposition~2.1]{nil99}.  Now, to see that $I \rtimes_{\alpha} \what{S}_f \coloneq \left(\varphi \rtimes \what{S}_f \right)\left(I \rtimes_{\alpha|} \what{S}_f \right)$ is actually an ideal of $A \rtimes_{\alpha} \what{S}_f$, we compute \begin{align*}
        \left(I \rtimes_{\alpha}  \what{S}_f\right)\left(A \rtimes_{\alpha}  \what{S}_f\right) &= \left(j^{\alpha}_A|_{I} \rtimes j^A\right)\left(\o\spn\left\{j^{\alpha|}_I(I)j^I\left(\what{S}_f\right)\right\} \right)\o\spn\left\{j^A\left(\what{S}_f\right)j^{\alpha}_A(A)\right\} \\&=\o\spn\left\{j^{\alpha}_A(I)j^A\left(\what{S}_f\right)\right\} \o\spn\left\{j^A\left(\what{S}_f\right)j^{\alpha}_A(A)\right\} \\&=\o\spn\left\{j^{\alpha}_A(I)j^A\left(\what{S}_f\right)j^{\alpha}_A(A)\right\} \\& = \o\spn\left\{j^{\alpha}_A(I)\o\spn\left\{j^A\left(\what{S}_f\right)j^{\alpha}_A(A)\right\}\right\} \\& =\o\spn\left\{j^{\alpha}_A(I)\o\spn\left\{j^{\alpha}_A(A)j^A\left(\what{S}_f\right)\right\}\right\} \\& =\o\spn\left\{j^{\alpha}_A(I)j^{\alpha}_A(A)j^A\left(\what{S}_f\right)\right\} \\& =\o\spn\left\{j^{\alpha}_A(I)j^A\left(\what{S}_f\right)\right\} \\&=I \rtimes_{\alpha} \what{S}_f.
    \end{align*}  Similarly we also obtain $\left(A \rtimes_{\alpha}  \what{S}_f\right)\left(I \rtimes_{\alpha}  \what{S}_f\right) = I \rtimes_{\alpha} \what{S}_f$, so we conclude $I \rtimes_{\alpha} \what{S}_f \lhd A \rtimes_{\alpha} \what{S}_f$.  To get that $I \rtimes_{\alpha} \what{S}_f \in \scr{I}_{\widehat{\alpha}}\left(A \rtimes_{\alpha} \what{S}_f \right)$, we again compute \begin{align*} \widehat{\alpha}\left(I \rtimes_{\alpha} \what{S}_f \right) &=\widehat{\alpha}\left( \o\spn\left\{j^{\alpha}_A(I)j^A\left(\what{S}_f\right)\right\}\right) \\& = \l(j^{\alpha}_A \otimes 1_{M\left(\what{S}_f\right)}\r) \rtimes \l(j^A \otimes \id_{\what{S}_f}\r) \circ \what{\Delta}_f\left( \o\spn\left\{j^{\alpha}_A(I)j^A\left(\what{S}_f\right)\right\}\right) \\&=\o\spn\left\{\left(j^{\alpha}_A(I) \otimes 1_{M\left(\what{S}_f\right)} \right)\left(j^A \otimes \id_{\what{S}_f} \right) \circ \what{\Delta}_f\left(\what{S}_f\right) \right\}.
    \end{align*} By Lemma \ref{lem10} it is sufficient to prove that \[\o\spn\left\{\widehat{\alpha}\left(I \rtimes_{\alpha} \what{S}_f\right)\left(1_{M\left(A \rtimes_{\alpha} \what{S}_f\right)} \otimes \what{S}_f \right)\right\} = \left(I \rtimes_{\alpha} \what{S}_f\right) \otimes \what{S}_f.\]  Indeed, we have \begin{align*}
        &\o\spn\left\{\widehat{\alpha}\left(I \rtimes_{\alpha} \what{S}_f\right)\left(1_{M\left(A \rtimes_{\alpha} \what{S}_f \right)}\otimes \what{S}_f \right)\right\} \\& = \o\spn\left\{\left(j^{\alpha}_A(I) \otimes 1_{M\left(\what{S}_f\right)} \right)\left(j^A \otimes \id_{\what{S}_f} \right)\circ \what{\Delta}_f\left(\what{S}_f\right)\left(1_{M\left(A \rtimes_{\alpha} \what{S}_f \right)} \otimes \what{S}_f \right) \right\} \\&= \o\spn\left\{\left(j^{\alpha}_A(I) \otimes 1_{M\left(\what{S}_f\right)} \right)\left(j^A \otimes \id_{\what{S}_f} \right) \left( \what{\Delta}_f\left(\what{S}_f\right)\left(1_{M\left(\what{S}_f\right)} \otimes \what{S}_f \right) \right)\right\} \\ 
        & =\o\spn\left\{\left(j^{\alpha}_A(I) \otimes 1_{M\left(\what{S}_f\right)} \right)\left(j^A \otimes \id_{\what{S}_f} \right) \left( \what{\Delta}_f\left(\what{S}_f\right)\left(1_{M\left(\what{S}_f\right)} \otimes \what{S}_f \right) \right)\right\} \\ \quad \quad &=\o\spn\left\{\left(j^{\alpha}_A(I) \otimes 1_{M\left(\what{S}_f\right)} \right)\left(j^A \otimes \id_{\what{S}_f} \right)\left(\what{S}_f \otimes \what{S}_f \right)\right\} \\ &= \o\spn\left\{j_A^{\alpha}(a)j^A(x) \otimes y \: | \: a \in I, x,y \in \what S_f\right\}  \\&=\left(I \rtimes_{\alpha} \what{S}_f \right) \otimes \what{S}_f.
    \end{align*} Therefore we obtain $I \rtimes_{\alpha} \what{S}_f \in \scr{I}_{\widehat{\alpha}}\left(A \rtimes_{\alpha} \what{S}_f \right)$ as claimed.  Now, consider the dual coaction $\widehat{\alpha}^n$.  Since $\Lambda_{\alpha}$ is onto $A \rtimes_{\alpha,r} \what S_r$, we have that $I \rtimes_{\alpha,r} \what S_r \coloneq \Lambda_{\alpha}\l(I \rtimes_{\alpha} \what S_f\r)  \lhd A \rtimes_{\alpha,r} \what S_r$.  Moreover, observe that \begin{align*}
        &\o\spn\l\{\widehat{\alpha}^n\left(I \rtimes_{\alpha,r} \what{S}_r\right)\l(1_{\M\l(A \rtimes_{\alpha,r} \what S_r \r)} \otimes \what S_f\r)\r\} \\ &\quad = \o\spn\l\{\widehat{\alpha}^n \circ \Lambda_{\alpha}\left(I \rtimes_{\alpha} \what{S}_f \right)\l(1_{\M\l(A \rtimes_{\alpha,r} \what S_r \r)} \otimes \what S_f\r)\r\} \\& \quad =\o\spn\l\{\left(\Lambda_{\alpha} \otimes \id_{\what{S}_f} \right) \circ \widehat{\alpha}\left(I \rtimes_{\alpha} \what{S}_f\right)\l(1_{\M\l(A \rtimes_{\alpha,r} \what S_r \r)} \otimes \what S_f\r)\r\} \\& \quad=\o\spn\l\{\left( \Lambda_{\alpha} \otimes \id_{\what{S}_f}\right)\left(\what{\alpha}\l(I \rtimes_{\alpha} \what S_f\r)\l(1_{\M\l(A \rtimes_{\alpha} \what S_f \r)} \otimes \what S_f\r)\right)\r\} \\ & \quad =\l(\Lambda_{\alpha} \otimes \id_{\what S_f}\r)\l(\o\spn\l\{\what{\alpha}\l(I \rtimes_{\alpha} \what S_f\r)\l(1_{\M\l(A \rtimes_{\alpha} \what S_f \r)} \otimes \what S_f\r)\r\} \r) \\ & \quad= \l(\Lambda_{\alpha} \otimes \id_{\what S_f} \r)\l(I \rtimes_{\alpha} \what S_f \otimes \what S_f\r) \\ & \quad = I \rtimes_{\alpha,r} \what S_r \otimes \what S_f
    \end{align*} proving that $I \rtimes_{\alpha,r} \what{S}_r \in \scr{I}_{\widehat{\alpha}^n}\left(A \rtimes_{\alpha,r} \what{S}_r\right)$ as claimed.  
\end{proof} 

Let $(C,\gamma), (D,\delta)$ be full $S$-coactions.  If $_{C}\sf{X}_D$ is a $(C,D)$-correspondence with $\gamma-\delta$ compatible coaction $\zeta$, we denote this data by $_{(C,\gamma)}(\sf{X},\zeta)_{(D,\delta)}$. 

\begin{lemma}[{\cite[Proposition~2.28]{fischer04}}]
\label{lem7}
    Let $_{(A,\alpha)}(\sf{X},\zeta)_{(B,\beta)}$ and $_{(B,\beta)}(\sf{Y},\eta)_{(C,\vartheta)}$ be $(A,B)$ and $(B,C)$-imprimitivity bimodules respectively, where $(A,\alpha)$, $(B,\beta)$ and $(C,\vartheta)$ are full $S$-coactions.  The map \[\zeta \: \sharp_B \: \eta \coloneq \Theta \circ (\zeta \otimes_B \eta)\] is a (continuous whenever $\zeta,\eta$ are) $\alpha-\vartheta$ compatible full correspondence coaction on $\sf{X} \otimes_B \sf{Y}$, where:
    \begin{enumerate}[label=(\alph*)]
        \item $\zeta \otimes_{B} \eta: \sf{X} \otimes_{B} \sf{Y} \longrightarrow M\left((\sf{X} \otimes_{\ext} S_f) \otimes_{B \otimes S_f} (\sf{Y} \otimes_{\ext} S_f) \right)$ is the $\alpha-\vartheta$ compatible correspondence homomorphism defined via \cite[Proposition~1.34]{enchilada}.
        \item $\Theta: (\sf{X} \otimes_{\ext} S_f) \otimes_{B \otimes S_f} (\sf{Y} \otimes_{\ext} S_f) \longrightarrow (\sf{X} \otimes_B \sf{Y}) \otimes_{\ext} S_f$ is the isomorphism defined via \cite[Lemma~2.12]{enchilada}. 
    \end{enumerate}
\end{lemma}

Now, let $(A,\alpha)$ be a maximal full $S$-coaction with $I \in \scr{I}_{\alpha}(A)$ with $\varphi: I \rtimes_{\alpha} \what S_f \hookrightarrow A \rtimes_{\alpha} \what S_f$ and $\psi: I \rtimes_{\alpha,r} \what S_r \hookrightarrow A \rtimes_{\alpha,r} \what S_r$ the canonical inclusions.  Then $\Phi_{\alpha}: A \rtimes_{\alpha} \what S_f \rtimes_{\what{\alpha}} S_f \longrightarrow A \otimes \K$ is an $\what{\what{\alpha}}-\epsilon$ equivariant isomorphism via Theorem \ref{thm2}.  Hence, we obtain the obvious imprimitivty bimodule \[(\sf{X}_1,\zeta_1) \coloneq \l(\p{}{(A \rtimes_{\alpha} \what S_f \rtimes_{\what\alpha} S_f,\what{\what{\alpha}})}{(A \otimes \K)}_{(A \otimes \K,\epsilon)}, \epsilon \r).\]  The ideal $I \rtimes_{\alpha} \what S_f \rtimes_{\what\alpha} S_f \coloneq (\varphi \rtimes S_f)\l(I \rtimes_{\alpha} \what S_f \rtimes_{\what\alpha|} S_f\r)$ is then sent to $I \otimes \K$ by the Rieffel correspondence since $\Phi_{\alpha}$ restricts to the canonical surjection of $I \rtimes_{\alpha} \what S_f \rtimes_{\what\alpha} S_f$.  Now, define the imprimitivty bimodule \[(\sf{X}_2,\zeta_2) \coloneq (\p{}{(A \otimes \K, \epsilon)}{(A \otimes \K)_{(A \otimes \K,\alpha \otimes_{*} \id_{\K})}}, \mf{w}_{23}(\alpha \otimes_{*} \id_{\K}))\] with the usual operations one places on $\p{}{D}{D}_D$ for a $\cst$-algebra $D$.  The Rieffel correspondence then sends $I \otimes \K$ to itself clearly.  Finally, one defines the usual imprimitivty bimodule \[(\sf{X}_3,\zeta_3) \coloneq (\p{}{(A \otimes \K, \alpha \otimes_{*} \id_{\K})}{(A \otimes \K)}_{(A,\alpha)}, \alpha \otimes_{*} \id_{L^2(\G)})\] that establishes the Morita equivalence between $A \otimes \K$ and $A$.  It's once again clear the Rieffel correspondence sends $I \otimes \K$ to $I$ so that when taking the $(A \otimes \K)$-balanced tensor product to obtain an $(A \rtimes_{\alpha} \what S_f \rtimes_{\what\alpha} S_f,A)$-imprimitivity bimodule \[(\sf{X},\zeta) \coloneq (\sf{X}_1 \otimes_{A \otimes \K} \sf{X}_2 \otimes_{A \otimes \K}, \sf{X}_3, \zeta_1 \: \sharp_{A \otimes \K} \: \zeta_2 \: \sharp_{A \otimes \K} \: \zeta_3)\] that the Rieffel correspondence sends $I \rtimes_{\alpha} \what S_f \rtimes_{\what\alpha} S_f$ to $I$.  Still in the maximal case, we have that $\Phi_{\alpha} \circ (\Lambda_{\alpha} \rtimes S_f)^{-1}: A \rtimes_{\alpha,r} \what{S}_r \rtimes_{\widehat{\alpha}^n} S_f \longrightarrow A \otimes \K$ is an $\widehat{\hat{\alpha}}^n-\epsilon$ equivariant isomorphism again via Theorem \ref{thm2}.  In particular, one yields the obvious imprimitivty bimodule \[(\sf{X}^n_1, \zeta_1) \coloneq (\p{}{(A \rtimes_{\alpha,r} \what{S}_r \rtimes_{\widehat{\alpha}^n} S_f , \what{\what{\alpha}}^n)}{(A \otimes \K)}_{(A \otimes \K,\epsilon)}, \epsilon).\] The Rieffel correspondence then sends $I \rtimes_{\alpha,r} \what{S}_r \rtimes_{\widehat{\alpha}^n} S_f \coloneq (\psi \rtimes S_f)\l(I \rtimes_{\alpha,r} \what{S}_r \r.$ $\l. \rtimes_{\widehat{\alpha}^n|} S_f\r)$ to $I \otimes \K$ since $\Lambda_{\alpha} \rtimes S_f\l(I \rtimes_{\alpha} \what S_f \rtimes_{\what\alpha} S_f\r) = I \rtimes_{\alpha,r} \what S_r \rtimes_{\what{\alpha}^n} S_f$.  Using $(\sf{X}_2,\zeta_2)$ and $(\sf{X}_3,\zeta_3)$ as above, we see that $I \otimes \K$ gets sent to $I$ via composing Rieffel correspondences.  Thus, we again may take an $(A \otimes \K)$-balanced tensor product to yield an $(A \rtimes_{\alpha,r} \what S_r \rtimes_{\what{\alpha}^n} S_f,A)$-imprimitivity bimodule \[(\sf{X}^n,\zeta) \coloneq (\sf{X}_1^n \otimes_{A \otimes \K} \sf{X}_2 \otimes_{A \otimes \K} \sf{X}_3, \zeta_1 \: \sharp_{A \otimes \K} \zeta_2 \: \sharp_{A \otimes \K} \: \zeta_3)\] where the Rieffel correspondence sends $I \rtimes_{\alpha,r} \what{S}_r \rtimes_{\widehat{\alpha}^n} S_f$ to $I$.  Finally, assume that $(A,\alpha)$ is normal so that we obtain an $(\what{\what{\alpha}})^n-\epsilon$ equivariant isomorphism $\Phi_{\alpha}^r: A \rtimes_{\alpha} \what S_f \rtimes_{\what\alpha,r} S_r \longrightarrow A \otimes \K$ and thus we obtain the obvious imprimitivty bimodule \[(\sf{X}_1^r, \zeta) \coloneq (\p{}{(A \rtimes_{\alpha} \what S_f \rtimes_{\what\alpha,r} S_r,(\what{\what{\alpha}})^n)}{(A \otimes \K)}_{(A \otimes \K,\epsilon)}, \epsilon)\] again via Theorem \ref{thm2}. Then the Rieffel correspondence sends $I \rtimes_{\alpha} \what S_f \rtimes_{\what\alpha,r} S_r \coloneq (\varphi \rtimes S_f)\l(I \rtimes_{\alpha} \allowbreak\what S_f \rtimes_{\what\alpha|,r} S_r\r)$ to $I \otimes \K$ using the fact that $\Phi_{\alpha}^r \circ \what{\Lambda}_{\what\alpha} = \Phi_{\alpha}$ for $\what{\Lambda}_{\what\alpha} : A \rtimes_{\alpha} \what S_f \rtimes_{\what\alpha} S_f \longrightarrow A \rtimes_{\alpha} \what S_f \rtimes_{\what\alpha,r} S_r$ the $\what S$-regular representation associated to $\what\alpha$.  Using the same imprimitivity bimodules $(\sf{X}_2,\zeta_2), (\sf{X}_3,\zeta_3)$ as above, we see that $I \otimes \K$ once again is sent to $I$ via composing the Rieffel correspondences.  Thus, we obtain a $(A \rtimes_{\alpha} \what S_f \rtimes_{\what{\alpha},r} S_r,A)$-imprimitivity bimodule \[(\sf{X}^r,\zeta) \coloneq (\sf{X}_1^r \otimes_{A \otimes \K} \sf{X}_2 \otimes_{A \otimes \K} \sf{X}_3, \zeta_1 \: \sharp_{A \otimes \K} \: \zeta_2 \: \sharp_{A \otimes \K} \: \zeta_3)\] such that the Rieffel correspondence sends $I \rtimes_{\alpha} \what S_f \rtimes_{\what\alpha,r} S_r$ to $I$.  Compatibility of the correspondence coactions $\zeta_1,\zeta_2,\zeta_3$ follows from \cite[Proof of Proposition~4.2]{mans}.  We have thus proven the following two propositions.

\begin{prop}
    \label{prop3} Let $(A,\alpha)$ be a full $S$-coaction with $I \in \scr{I}_{\alpha}(A)$.  
    \begin{enumerate}
        \item If $\alpha$ is maximal, then the ideals of $A \rtimes_{\alpha} \widehat{S}_f \rtimes_{\widehat{\alpha}} S_f$ and $A \times_{\alpha,r} \widehat{S}_r \rtimes_{\widehat{\alpha}^n} S_f$ associated to $I$ via the Rieffel correspondence are \begin{align*}
            I \rtimes_{\alpha} \what{S}_f \rtimes_{\widehat{\alpha}} S_f &\coloneq \varphi \rtimes S_f\left(I \rtimes_{\alpha} \what{S}_f \rtimes_{\widehat{\alpha}|} S_f\right) \: \text{and} \\ I \rtimes_{\alpha,r} \what{S}_r \rtimes_{\widehat{\alpha}^n} S_f &\coloneq \psi \rtimes S_f\left(I \rtimes_{\alpha,r} \what{S}_f \rtimes_{\widehat{\alpha}^n|} S_f\right)
        \end{align*}  where \begin{align*}
            &\varphi: I \rtimes_{\alpha} \what{S}_f  \hookrightarrow A \rtimes_{\alpha} \what{S}_f \: \text{and}  \\ &\psi: I \rtimes_{\alpha,r} \what{S}_f \hookrightarrow A \rtimes_{\alpha,r} \what{S}_f 
        \end{align*} are the canonical inclusions.
        \item If $\alpha$ is normal, then the ideal of $A \times_{\alpha} \what{S}_f \rtimes_{\widehat{\alpha},r} S_r$ associated to $I$ via the Rieffel correspondence is \begin{align*}
            I \rtimes_{\alpha} \what{S}_f \rtimes_{\widehat{\alpha},r} S_r &\coloneq \varphi \rtimes S_r\left(I \rtimes_{\alpha} \what{S}_f \rtimes_{\widehat{\alpha}|,r} S_r\right)
        \end{align*} where \begin{align*}
            &\varphi: I \rtimes_{\alpha} \what{S}_f  \hookrightarrow A \rtimes_{\alpha} \what{S}_f
        \end{align*} is the canonical inclusion.
    \end{enumerate}
\end{prop}

\begin{prop}
\label{prop4}
    Let $(A,\alpha)$ be a full $S$-coaction.   
    \begin{enumerate}[label=(\alph*)]
        \item If $\alpha$ is maximal, there are $A \rtimes_{\alpha} \what{S}_f \rtimes_{\widehat{\alpha}} S_f - A$ and $A \rtimes_{\alpha,r} \what{S}_r \rtimes_{\widehat{\alpha}^n} S_f - A$ imprimitivity bimodules $(\sf{X},\zeta)$ and $(\sf{X}^n,\zeta)$ with $\zeta$ an $\widehat{\widehat{\alpha}}-\alpha$ and $\widehat{\hat{\alpha}}^n-\alpha$ compatible correspondence coaction.   
        \item If $\alpha$ is normal, there is an $A \rtimes_{\alpha} \what{S}_f \rtimes_{\widehat{\alpha},r} S_r- A$ imprimitivty bimodule $(\sf{X}^r,\zeta)$ where $\zeta$ is an $\left(\widehat{\widehat{\alpha}}\right)^n-\alpha$ compatible correspondence coaction.   
    \end{enumerate} 
\end{prop}

\section{Main Result}
We now present the main result of the paper.  
\begin{theorem} \label{thm3}
    Let $(\H,V,U)$ be a weak Kac system with full duality arising from a regular LCQG $\G$, and moreover let $(A,\alpha)$ be a full $S$-coaction. 
    \begin{enumerate}[label=(\alph*)]
        \item If $\alpha$ is maximal, there exist lattice isomorphisms \begin{align*}
            \scr{I}_{\alpha}(A)& \longrightarrow \scr{I}_{\widehat{\alpha}}\left(A \rtimes_{\alpha} \what{S}_f\right) \: \: \text{and} \: \:\scr{I}_{\alpha}(A) \longrightarrow \scr{I}_{\widehat{\alpha}^n}\left(A \rtimes_{\alpha,r} \what{S}_r\right) \: \text{via} \\ &I \mapsto I \rtimes_{\alpha} \what{S}_f \quad \quad \quad \quad \text{and} \quad \quad \quad I \mapsto I \rtimes_{\alpha,r} \what{S}_r.
        \end{align*} 
        \item If $\alpha$ is normal, there exists a lattice isomorphism \begin{align*}
            \scr{I}_{\alpha}(A)& \longrightarrow \scr{I}_{\widehat{\alpha}}\left(A \rtimes_{\alpha} \what{S}_f\right) \: \: \text{via} \\& I \mapsto I \rtimes_{\alpha} \what{S}_f.
        \end{align*} 
    \end{enumerate}
\end{theorem}

\begin{proof} 
          \begin{enumerate}[label=(\alph*)]
              \item 
            Suppose that $\alpha$ is maximal.  Consider the `ladder' diagram

\begin{center}
    \begin{tikzcd}[sep=tiny, every label/.append style = {font={\footnotesize}}]
                                                                                                                                                       &  & \mathscr{I}_{\widehat{\widehat{\widehat{\alpha}}}}\left(A \rtimes_{\alpha} \what{S}_f \rtimes_{\widehat{\alpha}} S_f \rtimes_{\widehat{\widehat{\alpha}}} \what{S}_f\right)                                                                                                                \\
\mathscr{I}_{\widehat{\widehat{\alpha}}}\left(A \rtimes_{\alpha} \what{S}_f \rtimes_{\widehat{\alpha}} S_f\right) \arrow[rru, no head]            &  &                                                                                                                                                                                                                                                                                                     \\
                                                                                                                                                       &  & \mathscr{I}_{\widehat{\alpha}}\left(A \rtimes_{\alpha} \what{S}_f \right) \arrow[llu, no head] \arrow[uu, "\mathsf{Y}-\text{Ind}"'] \\
\mathscr{I}_{\alpha}(A) \arrow[rru, no head] \arrow[uu, "\mathsf{X}-\text{Ind}"] &  &                                                                                                                                                                                                                                                                                                    
\end{tikzcd}
\end{center} where $\sf{X}$ is the imprimitivity bimodule gotten from Proposition \ref{prop4}(a) applied to $(A,\alpha)$ and $\sf{Y}$ the corresponding imprimitivity bimodule gotten from Proposition \ref{prop4}(a) applied to the maximal dual coaction $(A \rtimes_{\alpha} \what S_f, \what\alpha)$. Note that the codomains for the restricted Rieffel bijections make sense via Proposition \ref{prop1}, and thus are lattice isomorphisms since Proposition \ref{prop4} endows $\sf{X},\sf{Y}$ with continuous correspondence $\what{\what{\alpha}}-\alpha$ and $\what{\what{\what{\alpha}}}-\what{\alpha}$ compatible coactions (and since the restriction of a lattice isomorphism to a sub-lattice is another lattice isomorphism). Now we define the obvious maps 
\begin{align*}
f&: I \mapsto I \rtimes_{\alpha} \what S_f\\ g&: J \mapsto J \rtimes_{\what{\alpha}} S_f \\ h&: L \mapsto L \rtimes_{\what{\what{\alpha}}} \what S_f.
\end{align*}
 yielding the diagram \begin{center}
     \begin{tikzcd}[sep=tiny]
                                                                                                                                                    &  & \mathscr{I}_{\widehat{\widehat{\widehat{\alpha}}}}\left(A \rtimes_{\alpha} \what{S}_f \rtimes_{\widehat{\alpha}} S_f \rtimes_{\widehat{\widehat{\alpha}}} \what{S}_f\right)                                                                                                            \\
\mathscr{I}_{\widehat{\widehat{\alpha}}}\left(A \rtimes_{\alpha} \what{S}_f \rtimes_{\widehat{\alpha}} S_f\right) \arrow[rru, "h"]             &  &                                                                                                                                                                                                                                                                                                 \\
                                                                                                                                                    &  & \mathscr{I}_{\widehat{\alpha}}\left(A \rtimes_{\alpha} \what{S}_f \right) \arrow[llu, "g"] \arrow[uu, "\mathsf{Y}-\text{Ind}"'] \\
\mathscr{I}_{\alpha}(A) \arrow[rru, "f"'] \arrow[uu, "\mathsf{X}-\text{Ind}"] &  &                                                                                                                                                                                                                                                                                                
\end{tikzcd} 
 \end{center} By Proposition \ref{prop2}, the maps $f,g,h$ have images contained in the invariant ideals.  We claim that $f,g,h$ are in fact lattice isomorphisms.  Indeed, by Proposition \ref{prop3} the diagram above commutes.  We then appeal to \cite[Lemma~3.1]{ladder} to get that $f,g,h$ are bijections.  To see that $f$ is in fact a lattice isomorphism, note $I \subset I'$ in $\scr{I}_{\alpha}(A)$ implies from Proposition \ref{prop2} that \begin{align*}
    I \rtimes_{\alpha} \what S_f &= \o\spn\l\{j_A^{\alpha}(I)j^A\l(\what S_f \r)\r\} \\ & \subset \o\spn\l\{j_A^{\alpha}(I')j^A\l(\what S_f\r)\r\} \\ &= I' \rtimes_{\alpha} \what S_f
\end{align*} so $f$ is an order preserving bijection of lattices.  Analogously one can do the same thing to $g$ so that  \[f^{-1} = \left(\mathsf{X}-\text{Ind}^{A \rtimes_{\alpha} \what{S}_f \rtimes_{\widehat{\alpha}} S_f}_A\right)^{-1} \circ g\] is an order preserving bijection of lattices.  However any bijection of lattices that is order preserving with order preserving inverse is a lattice isomorphism. Hence $f$ is a lattice isomorphism.  Now, we can repeat this same argument by instead considering the ladder diagram \begin{center}
    \begin{tikzcd}[sep=tiny, font={\footnotesize}]
                                                                                                                                        &  & \scr{I}_{\delta}\l(A \rtimes_{\alpha,r} \what S_r \rtimes_{\what{\alpha}^n} S_f \rtimes_{\what{\hat{\alpha}}^n,r} \what S_r\r)                                                                                                              \\
{\scr{I}_{\what{\hat{\alpha}}^n}\l(A \rtimes_{\alpha,r} \what S_r \rtimes_{\what{\alpha}^n} S_f\r)} \arrow[rru, no head]                &  &                                                                                                                                                                                                                                                             \\
                                                                                                                                        &  & {\scr{I}_{\what{\alpha}^n}\l(A \rtimes_{\alpha,r} \what S_r\r)} \arrow[llu, no head] \arrow[uu, "{\sf{Y}-\Ind}"'] \\
\scr{I}_{\alpha}(A) \arrow[uu, "{\sf{X}-\Ind}"] \arrow[rru, no head] &  &                                                                                                                                                                                                                                                            
\end{tikzcd}
\end{center} where $\delta$ is the dual coaction of $\what{\hat{\alpha}}^n$ on the reduced crossed product $A \rtimes_{\alpha,r} \what S_r \rtimes_{\what{\alpha}^n} S_f \rtimes_{\what{\hat{\alpha}}^n,r} \what S_r$.  The Morita equivalence on the right vertical rung is obtained via normal duality applied to the $\cst$-algebra $\l(A \rtimes_{\alpha,r} \what S_r, \what{\alpha}^n\r)$.  Hence, by applying Proposition \ref{prop4} we obtain the imprimitivity bimodules $\sf{X},\sf{Y}$ needed for the calculation.  As before, note that by Proposition \ref{prop1}, the restrictions of the Rieffel bijections above are in fact lattice isomorphisms.  We again define obvious maps $f',g',h'$ by sending an invariant ideal to its crossed product. As before, $f',g',h'$ have images contained inside the invariant ideals by Proposition \ref{prop2}, and the diagram commutes by Proposition \ref{prop3} giving us bijections of lattices via \cite[Lemma~3.1]{ladder}.  Just as we showed $f$ was an order preserving bijection of lattices, so is $f'$.

\item If $\alpha$ is normal, we obtain another completely analogous ladder diagram \begin{center}
    \begin{tikzcd}[sep=tiny, font={\footnotesize}]
                                                                                                                                                          &  & {\mathscr{I}_{\gamma}\left(A \rtimes_{\alpha} \what{S}_f \rtimes_{\widehat{\alpha},r} S_r \rtimes_{\left(\widehat{\widehat{\alpha}}\right)^n} \what{S}_f\right)}                                                                                                                           \\
{\mathscr{I}_{\left(\widehat{\widehat{\alpha}}\right)^n}\left(A \rtimes_{\alpha} \what{S}_f \rtimes_{\widehat{\alpha},r} S_r\right)} \arrow[rru, no head] &  &                                                                                                                                                                                                                                                                                            \\
                                                                                                                                                          &  & \mathscr{I}_{\widehat{\alpha}}\left(A \rtimes_{\alpha} \what{S}_f\right) \arrow[llu, no head] \arrow[uu, "{\sf{Y}-\Ind}"'] \\
\mathscr{I}_{\alpha}(A) \arrow[rru, no head] \arrow[uu, "{\sf{X}-\Ind}"]                  &  &                                                                                                                                                                                                                                                                                           
\end{tikzcd}
\end{center} where $\gamma$ is the dual coaction of $\left(\widehat{\widehat{\alpha}}\right)^n$ and where $\sf{X},\sf{Y}$ are the appropriate imprimitivity bimodules equipped with compatible correspondence coactions from Proposition \ref{prop4}.  Note on the right vertical rung, the Morita equivalence is via maximal duality since the dual coaction $\widehat{\alpha}$ is maximal.   Again by Proposition \ref{prop1}, the Rieffel lattice isomorphisms restrict to lattice isomorphisms on the appropriate sub-lattices given above.  The remainder of the proof is completely analogous to the above calculations.  
\end{enumerate}
\end{proof}

\end{document}